\documentclass[a4paper,11pt]{scrartcl}
\addtokomafont{disposition}{\rmfamily}

\usepackage[left=2.5cm,right=2.5cm,top=2.5cm,bottom=2.5cm]{geometry}

\usepackage{amsthm}
\usepackage{amsmath}
\usepackage{amssymb}
\usepackage{mathrsfs} % Added by TF
\usepackage{graphicx}
\usepackage[dvipsnames]{xcolor}
\usepackage{stackengine}
\usepackage{placeins}
\usepackage{arydshln}
\usepackage{bbm}
\usepackage{textcomp}
\usepackage{enumitem}
\usepackage{mathtools}
\usepackage{tabularx}
\usepackage{booktabs}
\usepackage{rotating}

\usepackage{color}
\usepackage{framed}
\usepackage{comment}
\definecolor{shadecolor}{gray}{0.9}

\usepackage{dsfont} % Added by TF
\usepackage{caption} % Added by TF
\usepackage{subcaption} % Added by TF
\usepackage{graphicx}
\usepackage{pdfpages}
\usepackage{url}
\usepackage[english]{babel}
\usepackage{setspace}
\usepackage{dashrule}
\usepackage{tikz}
\usetikzlibrary{automata, positioning, arrows}
\tikzset{
        ->,  % makes the edges directed
        node distance=5.5cm, % specifies the minimum distance between two nodes. Change if n
        every state/.style={thick, fill=gray!10}, % sets the properties for each ÕstateÕ n
        initial text=$ $, % sets the text that appears on the start arrow
        }

\usepackage{xr-hyper} 
\usepackage[
breaklinks,
colorlinks=true,
pagebackref=false,
pdfauthor={},%
pdftitle={},% 
citecolor=blue,
linkcolor=blue,
urlcolor=blue,
filecolor=blue
]{hyperref}      
\usepackage{cleveref}
\usepackage[authoryear,round]{natbib}

\specialcomment{extra}{\begin{shaded}}{\end{shaded}}

\theoremstyle{plain}  %default 
\newtheorem{theorem}{Theorem}%[section] 

\theoremstyle{definition} 
\newtheorem{definition}[theorem]{Definition}
 
\newtheorem{exmp}[theorem]{Example}
\newtheorem{rem}[theorem]{Remark}

\newtheoremstyle{assumption}% ?name? 
{3pt}% ?Space above? 
{3pt}% ?Space below?
{}% ?Body font?
{}% ?Indent amount?1
{\bf}% ?Theorem head font?
{.}% ?Punctuation after theorem head?
{.5em}% ?Space after theorem head?2
{\thmname{#1} (\thmnote{#3}\thmnumber{#2})}% ?Theorem head spec (can be left empty, meaning ÔnormalÕ)?

\theoremstyle{assumption}
\newtheorem{ass}{Assumption}

\theoremstyle{remark}

\newcommand{\E}{\mathbb{E}}

\newcommand{\conv}{\operatorname{conv}}
\newcommand{\ES}{\operatorname{ES}}

\newcommand{\CoVaR}{\operatorname{CoVaR}}

\newcommand{\F}{\mathcal{F}}
\newcommand{\R}{\mathbb{R}}
\newcommand{\A}{\mathsf{A}}
\renewcommand{\O}{\mathsf{O}}
\newcommand{\one}{\mathds{1}}
\newcommand{\interior}{\operatorname{int}}

\newcommand{\VaR}{\operatorname{VaR}}

\newcommand{\tto}{\twoheadrightarrow}

\renewcommand{\rm}{\normalfont \rmfamily}
\renewcommand{\bf}{\normalfont \bfseries}

\def\be{\begin{equation} \label}
\def\ee{\end{equation}}

\usepackage[colorinlistoftodos,textsize=tiny]{todonotes}
\newcommand{\Comments}{1}
\newcommand{\mynote}[2]{\ifnum\Comments=1\textcolor{#1}{#2}\fi}
\newcommand{\mytodo}[2]{\ifnum\Comments=1%
  \todo[linecolor=#1!80!black,backgroundcolor=#1,bordercolor=#1!80!black]{#2}\fi}

\ifnum\Comments=1               % fix margins for todonotes
\fi

\ifnum\Comments=1               % fix margins for todonotes
\fi

\ifnum\Comments=1               % fix margins for todonotes
\fi
\allowdisplaybreaks

\begin{document}

\title{Osband's Principle for Identification Functions}
\author{Timo Dimitriadis\thanks{Heidelberg University, Alfred Weber Institute of Economics, Bergheimer Str.\ 58, 69115 Heidelberg, Germany and
		Heidelberg Institute for Theoretical Studies, 69118 Heidelberg, Germany, e-mail: \href{mailto:timo.dimitriadis@awi.uni-heidelberg.de}{timo.dimitriadis@awi.uni-heidelberg.de}}
	\and Tobias Fissler\thanks{Vienna University of Economics and Business (WU), Department of Finance, Accounting and Statistics, Welthandelsplatz 1, 1020 Vienna, Austria, 
		e-mail: \href{mailto:tobias.fissler@wu.ac.at}{tobias.fissler@wu.ac.at}} \and Johanna Ziegel\thanks{University of Bern, Department of Mathematics and Statistics, Institute of Mathematical Statistics and Actuarial Science, Alpeneggstrasse 22, 3012 Bern, Switzerland, 
		e-mail: \href{mailto:johanna.ziegel@stat.unibe.ch}{johanna.ziegel@stat.unibe.ch}}
}
\maketitle

\begin{abstract}
\noindent
\textbf{Abstract.}
Given a statistical functional of interest such as the mean or median, a (strict) identification function is zero in expectation at (and only at) the true functional value.
Identification functions are key objects in forecast validation, statistical estimation and dynamic modelling.
For a possibly vector-valued functional of interest, we fully characterise the class of (strict) identification functions subject to mild regularity conditions.
\end{abstract}
\\
\noindent
\textit{MSC2020 classes:}
62C07; 62F10; 62J20

\section{Introduction and informal statement of main result}
\label{sec:intro}

\onehalfspacing

Consider a statistical functional $T$ of the random variable $Y \sim F$, that is, a 
%real-valued 
mapping $F\mapsto T(F)$, such as the mean or the median.
In the theory of forecast validation, a corresponding strict identification function $V(x,y)$ takes the forecast $x$ and the realisation $y$ of $Y$ as arguments and its expectation with respect to $Y \sim F$ is zero if and only if $x$ equals the true functional value $T(F)$.
% They generalize the classical forecast error $x-y$ for mean forecasts and the hit sequence of quantiles.
This defining property makes identification functions a central tool in forecast validation through calibration tests \citep{NoldeZiegel2017}, often referred to as backtests in finance, and to forecast rationality (or optimality) tests in economics \citep{EKT2005, DimiPattonSchmidt2019}.
Furthermore, these functions are fundamental to zero (Z) or generalised method of moments (GMM) estimation \citep{Huber1967, Hansen1982, NeweyMcFadden1994}, where they are often called moment functions or moment conditions.
However, their statistical applications go much beyond these two fields and among others, they influence dynamic modelling through generalised autoregressive score (GAS) models \citep{Creal2013}, isotonic regression estimates \citep{JordanMuehlemannZiegel2019}, or the derivation of  anytime valid sequential tests \citep{casgrain2022anytime}.
A complete understanding of the full class of (strict) identification functions for a given functional is crucial in these applications. 
Our main contribution, Theorem \ref{theorem}, provides such a full characterisation result.

In the jargon of decision theory \citep{Gneiting2011}, 
the quantity of interest $Y$ attains values in an \emph{observation domain} $\O\subseteq \R^d$, which is equipped with the Borel-$\sigma$-algebra.
The class of potential probability distributions $F$ of $Y$ is denoted by $\F$.
Forecasts are elements of an \emph{action domain} $\A\subseteq \R^k$.
Formally, the functional of interest $T$ is a potentially set-valued mapping from $\F$ to $\A$, denoted by $T:\F\tto \A$, where the notation $\tto$ indicates that the values of $T$ are subsets of $\A$, with the convention that we identify point-valued functionals such as the mean with the singleton containing this value.
For $\O=\A=\R$, prime examples for $T$ are the mean or the $\alpha$-quantile $q_\alpha(F) = \{x\in\R \,|\, \lim_{t\uparrow x}F(t)\le \alpha \le F(x)\}$, $\alpha\in(0,1)$, where the latter is interval-valued.
Prime examples for multivariate functionals are the mean-functional in case of multivariate observations ($\O=\A=\R^k$). For univariate observations, examples are multiple quantiles at different levels, the pair (mean, variance) with the natural action domain $\A = \R\times [0,\infty)$ or the pair consisting of the quantile and the Expected Shortfall (ES) at the same level with natural action domain $\A = \{(x_1,x_2)\in\R^2 \,|\, x_1\ge x_2\}$, see Examples \ref{exmp:mean,var} and \ref{exmp:VaR,ES} for details. 
To present the formal definition of an identification function $V:\A\times\O\to\R^k$, let us introduce the convention that $V$ is called \emph{$\F$-integrable} if for each of its components $V_i$ the integral $\int_\O V_i(x,y)\,\mathrm{d}F(y)$ exists and is finite for all $x\in\A$ and $F\in\F$.
Moreover, we shall use 
the shorthand $\bar V(x,F) = \int_\O V(x,y)\,\mathrm{d}F(y)$ for any $x\in\A$, $F\in\F$, where the integral is understood componentwise.

\begin{definition}[Identification function and identifiability]\label{defn:identifiability} 
	$ $ \vspace{-0.2cm}
\begin{enumerate}[label=(\roman*)]
	\item
	An $\F$-integrable map $V: \A\times\O\to\R^k$ is an \emph{$\F$-identification function} for a functional $T: \F\twoheadrightarrow \A\subseteq\R^k$ if 
	for all $x\in\A$ and for all $F\in\F$ 
	\[
	x\in T(F) \implies \bar V(x,F)=0.
	\]
	\item
	An $\F$-integrable map $V: \A\times\O\to\R^k$ is a \emph{strict $\F$-identification function} for a functional $T: \F\twoheadrightarrow \A\subseteq\R^k$ if 
	for all $x\in\A$ and for all $F\in\F$ 
	\[
	x\in T(F) \iff \bar V(x,F)=0.
	\]
	\item
	A functional $T: \F\twoheadrightarrow \A\subseteq\R^k$ is called \emph{$\F$-identifiable} if there exists a strict $\F$-identification function for it.
	\end{enumerate}
\end{definition}
On the class of distributions on $\R$ with a finite mean, $\F^1(\R)$, the mean is identifiable with strict $\F^1(\R)$-identification function $V(x,y)= x-y$. 
Likewise, the $\tau$-expectile, $\tau\in(0,1)$, possesses a strict $\F^1(\R)$-identification function $V(x,y) = 2|\one\{y\le x\} - \tau|(x-y)$. 
On the class $\F_\alpha(\R)$ of distributions on $\R$ such that there exists an $x$ with $F(x) = \alpha$, the $\alpha$-quantile admits the strict $\F_\alpha(\R)$-identification function $V(x,y) = \one\{y\le x\}-\alpha$.
Functionals failing to be identifiable on practically relevant classes of distributions are the variance and Expected Shortfall. 
On such classes $\F$, both of them violate the selective convex level sets property, which is necessary for identifiability \citep{Osband1985, FisslerHlavinovaRudloff2019Theory}.\footnote{$T$ satisfies the selective convex level sets property of $\F$ if for any $F,G\in\F$ and for any $\lambda\in(0,1)$ such that $(1-\lambda)F + \lambda G\in\F$ it holds that $T(F)\cap T(G) \subseteq T((1-\lambda)F + \lambda G)$.}
However, the pairs (mean, variance) and (quantile, ES) turn out to be identifiable with corresponding two-dimensional strict identification functions, see Examples \ref{exmp:mean,var} and \ref{exmp:VaR,ES}. 

Regarding the flexibility of the class of identification functions, the following observation is immediate: If 
$V(x,y)$ is a strict $\F$-identification function for $T:\F \twoheadrightarrow \A\subseteq \R^k$, it can be multiplied with any $\R^{k\times k}$-valued function $h(x)$ of full rank and remains a strict identification function for $T$.
Intriguingly, Theorem \ref{theorem} formally states that, subject to mild regularity conditions, the reverse is also true, and the entire class of strict identification functions is given by 
\begin{equation}
\label{eq:characterisation}
\big\{ h(x) V(x,y)\,|\, h: \A \to\R^{k\times k}, \ \det (h(x))\neq0 \ \text{for all } x \in\A \big\}.
\end{equation}

Besides its theoretical appeal, this characterisation result opens the way for diverse applications.
First, it can be used to optimise power of (conditional) calibration (forecast rationality or optimality) tests studied in \cite{NoldeZiegel2017}.
It is further related to efficient Z- or GMM-estimation based on conditional moment conditions in the sense of \citet{Chamberlain1987} and \cite{Newey1993}, where the matrix $h$ is submerged in the choice of an optimal instrument matrix; see Theorem 3.1 and especially Remark 3.2 in \cite{DFZ2020} for details.
Based on the choice of an identification function (called score by these authors) as their forcing variable, dynamic GAS models of \cite{Creal2013} determine an autoregressive model structure for a corresponding functional of interest that nests classical ARMA and GARCH models for the mean and variance.
In these models, the so-called scaling matrix takes the place of the matrix $h$ and, as already called for by \citet[p.~779]{Creal2013}, this choice ``warrants separate inspection''.

The following examples discuss interesting applications of our characterisation result in \eqref{eq:characterisation} to vector-valued functionals.
\begin{exmp}[Mean and variance]
\label{exmp:mean,var}
The pair (mean, variance) is identifiable on the class $\F^2(\R)$ of distributions with finite variance with the two-dimensional strict $\F^2(\R)$-identification function
\[
V(x_1,x_2,y) = 
\begin{pmatrix}
x_1 - y\\
x_2 - (y-x_1)^2
\end{pmatrix}.
\]
One can use the characterisation result \eqref{eq:characterisation} to produce a multitude of other strict $\F^2(\R)$-identification functions.
Motivated by the decomposition of the variance into the difference of the second moment the squared expectation, a comparably intuitive one is 
\begin{equation}
\label{eq:mean,var}
V'(x_1,x_2,y) = 
\begin{pmatrix}
x_1 - y\\
x_2 + x_1^2 - y^2
\end{pmatrix},
\end{equation}
which arises by choosing the full rank matrix $h(x_1,x_2) = 
\left(
\begin{smallmatrix}
1 &  0 \\
2x_1 & 1
\end{smallmatrix}\right)$.
\end{exmp}

\begin{exmp}[Quantile and ES]
\label{exmp:VaR,ES}
In financial mathematics, Value-at-Risk at level $\alpha\in(0,1)$ ($\VaR_\alpha$) denotes the lower $\alpha$-quantile, $\VaR_\alpha(F) = \inf q_\alpha(F) = \inf\{x \in\R \,|\, \alpha \le F(x)\}$.
Then, the ES at level $\alpha\in(0,1)$ of a distribution $F$ is formally defined as 
\begin{equation}
	\label{eq:Def_ES}
	\ES_\alpha(F) 
	= \frac{1}{\alpha}\int_0^\alpha\VaR_\beta(F)\,\mathrm{d}\beta 
	= 
	\frac{1}{\alpha}\int y\one\{y\le \VaR_\alpha(F)\} \,\mathrm{d}F(y) - \frac{\VaR_\alpha(F)}{\alpha}\big(F(\VaR_{\alpha}(F)) - \alpha\big).
\end{equation}
On any subclass of $\F_\alpha(\R)$ where $\ES_\alpha$ is finite, e.g.~on $\F_\alpha(\R)\cap\F^1(\R)$, there is the following strict identification function for $(q_\alpha,\ES_\alpha$)
\[
V(x_1,x_2,y) = 
\begin{pmatrix}
\one\{y\le x_1\} - \alpha \\
x_2 - \frac{y}{\alpha}\one\{y\le x_1\}
\end{pmatrix},
\]
where the second component naturally corresponds to a truncated expectation. 
Applying \eqref{eq:characterisation} with the full rank matrix $h(x_1,x_2) = 
\left(
\begin{smallmatrix}
1 &  0 \\
x_1/\alpha & 1
\end{smallmatrix}\right)$, one obtains the alternative strict identification function 
\begin{equation}
\label{eq:quantile,ES}
V'(x_1,x_2,y) = 
\begin{pmatrix}
\one\{y\le x_1\} - \alpha \\
x_2 - \frac{y}{\alpha}\one\{y\le x_1\} + \frac{x_1}{\alpha}(\one\{y\le x_1\} - \alpha)
\end{pmatrix}.
\end{equation}
The advantage of $V'$ over $V$ is that when evaluating $V'$ on a discontinuous distribution with $F(\VaR_\alpha(F))>\alpha$, even though the first components of $V$ and $V'$ fail to be an identification function for $q_\alpha$,\footnote{
	To obtain a better understanding of identifiability for the possibly set-valued $\alpha$-quantile and its lower endpoint $\VaR_\alpha$, one can distinguish three cases.
	First, if $F$ is strictly increasing and continuous at its $\alpha$-quantile, the latter is singleton-valued and $V(x,y) = \one\{y\le x\}-\alpha$ is a strict identification function both for $q_\alpha$ and for $\VaR_\alpha$.
	Second, if $F$ is flat at its set-valued $\alpha$-quantile, $V$ is still a strict identification function for the set-valued $q_\alpha$, but it is only a (non-strict) identification function for the singleton-valued $\VaR_\alpha$.
	Third, if $F$ is discontinuous at $\VaR_\alpha(F)$ such that $F(\VaR_\alpha(F))>\alpha$ (that is, if $F\notin \F_\alpha(\R)$), neither $q_\alpha$ nor $\VaR_\alpha$ are identified by $V$.}
the second component of $V'$ still vanishes in expectation when plugging in the correct values for $q_\alpha(F)$ and $\ES_\alpha(F)$ for $x_1$ and $x_2$.
Intuitively, the second component of $V'$ adds a correction term corresponding to the one on the right-hand side of \eqref{eq:Def_ES}.
The choice \eqref{eq:quantile,ES} is already utilised by \citet[Equation (4)]{DimiBayer2019} for Z-estimation of a joint quantile and ES regression model
and naturally shows up in consistent scoring functions for $(q_\alpha, \ES_\alpha)$, see \citet[Corollary 5.5]{FisslerZiegel2016}.
Finally notice that the $\ES_\alpha(F)$ is sometimes also defined as the upper average quantile over $\VaR_\beta$ with $\beta \in (\alpha,1)$. Then, our results apply \textit{mutatis mutandis}.
\end{exmp}

\section{Formal statement of main result} 

The assertion of Theorem \ref{theorem}, and in particular its proof, parallels Osband's principle for consistent scoring functions \citet[Theorem 3.2]{FisslerZiegel2016}, see also \cite{Osband1985, Gneiting2011}.
Up to our knowledge, the assertion has first been stated in the PhD thesis \citet[Proposition 3.2.1]{Fissler2017}. 
We need the following assumptions.

\begin{ass}\label{ass:V1}
Let $\F$ be a convex class of distributions on $\O$ such that for every $x\in \interior (\A) \subseteq \R^k$ there are $F_1,\ldots, F_{k+1}\in\F$ satisfying 
\(
0\in \interior\big(\conv\big(\{ \bar V(x,F_1), \ldots, \bar V(x, F_{k+1})\}\big)\big)\,,
\)
where for any set $B\subseteq \R^k$, $\interior (B)$ denotes the interior of $B$ and $\conv (B)$ denotes the convex hull of $B$.
\end{ass}

\begin{ass}\label{ass:F1}
For every $y\in\R^d$ there exists a sequence $(F_n)_{n\in\mathbb N}$ of distributions $F_n\in\F$ that converges weakly to the Dirac-measure $\delta_y$ and a compact set $K\subset \R^d$ such that the support of $F_n$ is contained in $K$ for all $n$.
\end{ass}

\begin{ass}\label{ass:VS1}
Suppose that for Lebesgue almost all $x\in\interior(\A)$ the maps $V(x,\cdot)$ and $V'(x,\cdot)$ are locally bounded.
Moreover, suppose that the complement of the set 
\[
C:= \{(x,y)\in \interior(\A)\times \O \,|\,V(x,\cdot) \text{ and } V'(x,\cdot) \text{ are continuous at the point $y$}\}
\]
has $(k+d)$-dimensional Lebesgue measure zero.
\end{ass}

Assumptions \eqref{ass:V1}, \eqref{ass:F1}, and \eqref{ass:VS1} basically correspond to Assumptions (V1), (F1), and (VS1) in \cite{FisslerZiegel2016}, respectively. 
Assumption \eqref{ass:V1} ensures that the class $\F$ is sufficiently rich, implying in particular the surjectivity of $T$ onto $\interior(\A)$ and the fact that there are no redundancies in $V$ in the sense that all its components are needed; see Remark \ref{rem:discussion assumption} for some further comments.
Assumptions \eqref{ass:F1} and \eqref{ass:VS1} ensure that $V(x,y)$ can be approximated by a sequence of integrals $\bar V(x,F_n)$.

\begin{theorem}
\label{theorem}
Let $T: \F\twoheadrightarrow \A\subseteq\R^k$ be a functional with a strict $\F$-identification function $V: \A\times \O\to\R^k$. Then the following two assertions hold:
\begin{enumerate}[label=\rm(\roman*)]%[\em (i)]
\item
If $h\colon \A \to\R^{k\times k}$ is a matrix-valued function with $\det(h(x))\neq0$ for all $x\in\A$, then $V'(x,y)= h(x)V(x,y)$ is also a strict $\F$-identification function for $T$.
\item
Let $V$ satisfy Assumption \eqref{ass:V1} and let 
$V': \A\times \O\to\R^k$ be an $\F$-identification function for $T$. Then there is a matrix-valued function $h: \interior(\A)\to\R^{k\times k}$ such that 
\[%\be{eq:Osband identification function}
\bar V'(x,F) = h(x)\bar V(x,F) 
\]%\ee
for all $x\in\interior(\A)$ and for all $F\in\F$. 

If $V'$ is a strict $\F$-identification function for $T$ and it also satisfies Assumption \eqref{ass:V1}, then additionally $\det(h(x))\neq0$ for all $x\in\interior(\A)$. 
If the integrated identification functions $\bar V(\cdot,F)$ and $\bar V'(\cdot,F)$ are continuous, then also $h$ is continuous, which implies that either $\det(h(x))>0$ for all $x\in\interior(\A)$ or $\det(h(x))<0$ for all $x\in\interior(\A)$. 

Moreover, if $\F$ satisfies Assumption \eqref{ass:F1} and $V$, $V'$ satisfy Assumption \eqref{ass:VS1} it even holds that 
\begin{equation}
\label{eq:Osband identification function ptw}
V'(x,y) = h(x) V(x,y)
\end{equation}
for Lebesgue almost all $(x,y)\in \interior(\A)\times \O$.
\end{enumerate}
\end{theorem}

\begin{proof}[Proof of Theorem \ref{theorem}]
Part (i) is a direct consequence of the linearity of the expectation. For (ii), the proof of the existence of $h$ follows along the lines of Theorem 3.2 in \cite{FisslerZiegel2016}. One just needs to replace $\nabla \bar S(x,F)$ with $\bar V'(x,F)$. 
If $V'$ satisfies Assumption \eqref{ass:V1} as well, one directly obtains that $h$ must have full rank on $\interior(\A)$ by exchanging the roles of $V$ and $V'$.
If the expected identification functions are both continuous, the continuity of $h$ follows again exactly like in the proof of Theorem 3.2 in \cite{FisslerZiegel2016}. \\
For the pointwise assertion \eqref{eq:Osband identification function ptw}, 
consider $(x,y)\in \interior(\A)\times \O$ such that both $V(x,\cdot)$ and $V'(x,\cdot)$ are continuous at $y$. (Due to Assumption \eqref{ass:VS1}, this holds for Lebesgue almost all $(x,y)$.)
Let $(F_n)_{n\in\mathbb N}\subseteq \F$ be a sequence as specified in Assumption \eqref{ass:F1}. That is, $(F_n)_{n\in\mathbb N}$ converges weakly to $\delta_y$ and the supports of all $F_n$ are contained in some compact set $K\subset \R^d$. 
We claim that $\bar V(x,F_n)$ and $\bar V'(x,F_n)$ converge to $V(x,y)$ and $V'(x,y)$, respectively, providing the arguments for the former convergence only.
By Skorohod's theorem, there is a sequence of random variables $(\xi_n)_{n\in\mathbb N}$ on some probability space with distributions $F_n$, such that $\xi_n$ converges to $y$ almost surely.
By the continuous mapping theorem, $V(x,\xi_n)$ converges to $V(x,y)$ almost surely. 
Since $V(x,\cdot)$ is assumed to be locally bounded and since $\xi_n\in K$ almost surely, $V(x,\xi_n)$ is bounded almost surely.
Hence, we can apply the dominated convergence theorem to conclude that 
$\bar V(x,F_n) = \E V(x,\xi_n)\to V(x,y)$.
\end{proof}

\begin{rem}
\label{rem:discussion assumption}
For part (i) of Theorem \ref{theorem}, no surjectivity assumption is necessary. In fact, the identification functions at \eqref{eq:mean,var} and \eqref{eq:quantile,ES} are also strict identification functions for (mean, variance) and $(q_\alpha, \ES_\alpha)$, respectively, when considering the action domain $\A=\R^2$.
However, it is obvious that part (ii) of Theorem \ref{theorem} cannot hold without a surjectivity assumption. In fact, $V''(x_1,x_2,y) = V'(x_1,x_2,y) \one\{x_2\ge0\} + \one\{x_2<0\}$ would also be a strict identification function for (mean, variance) on the action domain $\R^2$.

On the other hand, also the richness, in particular, the convexity of $\F$ are needed. Just recall that on the class of symmetric distributions with strictly increasing distribution function, the mean and the median coincide. Hence, both $V(x,y) = x-y$ and $V'(x,y) = \one\{y\le x\} - 1/2$ are strict identification functions, but do not fulfil \eqref{eq:Osband identification function ptw}. 
The reason is that the class of symmetric distributions fails to be convex, unless all distributions have the same mean, in which case the interior of the action domain would be empty under surjectivity.
\end{rem}

\begin{rem}
	One may wonder about the flexibility concerning the dimension of an identification function. Suppose that $V(x,y)$ is a strict $\F$-identification function for some functional $T$, which takes values in $\R^k$.
	Clearly, for any matrix-valued function $h(x)\in \R^{\ell\times k}$ where possibly $\ell\neq k$, the product $V'(x,y) = h(x)V(x,y)$ is an $\F$-identification function for $T$. If $\ell>k$ and the rank of $h(x)$ is $k$ for all $x$, $V'$ is still a strict $\F$-identification function. 
	However, $V'$ will not satisfy Assumption \eqref{ass:V1}, thus containing redundancies (in fact, the easiest way to construct such a $V'$ is by simply copying some components of $V$).
	On the other hand, if $\ell<k$, the proof of Theorem \ref{theorem} (ii) implies that $V'$ cannot be a strict $\F$-identification function.
	
	The latter statement can be exemplified by considering the systemic risk measure $\CoVaR_{\alpha|\beta}$, which, given a two-dimensional observation $(Y_1, Y_2)$, it is defined as the $\VaR_\alpha$ of the conditional distribution of $Y_2$, given that $Y_1$ exceeds its $\VaR_\beta$.
	Then, the pair $(\VaR_{\beta}, \CoVaR_{\alpha|\beta})$ is identifiable on the class of absolutely continuous distributions with positive density on $\R^2$ with a corresponding strict identification function 
	\[
	V(x_1,x_2,y_1,y_2) = \begin{pmatrix}
	\one\{x_1\le y_1\} - \beta\\
	\one\{x_1>y_1\}\big(\one\{x_2\le y_2\} - \alpha\big)
	\end{pmatrix},
	\]
	see \citet[Theorem 4.2]{FisslerHoga2021}. Due to the argument above, the one-dimensional identification function 
	\[
	V'(x_1,x_2,y_1,y_2) = \one\{x_1>y_1\}\one\{x_2>y_2\} - (1-\alpha)(1-\beta)
	\]
	suggested in \cite{Banulescu-RaduETAL2021}
	cannot be a strict identification function for $(\VaR_{\beta}, \CoVaR_{\alpha|\beta})$ on the class of absolutely continuous distributions with positive density, see \citet[Remark 4.3]{FisslerHoga2021}.
\end{rem}

\section*{Acknowledgements}
T.~Dimitriadis gratefully acknowledges support of the German Research Foundation (DFG) through grant number 502572912 and of the Heidelberg Academy of Sciences and Humanities. J.~Ziegel gratefully acknowledges support of the Swiss National Science Foundation.
We are very grateful to Jana Hlavinov\'a for a careful proofreading and valuable feedback on an earlier version of this paper.

% \small
\singlespacing
\setlength{\bibsep}{2pt plus 0.3ex}
\def\bibfont{\small}

\bibliographystyle{apalike}
\bibliography{biblio_OsbandId}

\end{document}